\documentclass[11pt]{amsart}
\usepackage[margin=1.35in]{geometry}
\usepackage{amscd,amsmath,amsxtra,amsthm,amssymb,stmaryrd,xr,mathrsfs,mathtools,enumerate,commath, comment, enumitem}
\usepackage{stmaryrd}
\usepackage{xcolor}
\usepackage{commath}
\usepackage{comment}
\usepackage{tikz-cd}
\usepackage{longtable} 
\usepackage{pdflscape} 
\usepackage{booktabs}
\usepackage{hyperref}
\definecolor{vegasgold}{rgb}{0.77, 0.7, 0.35}
\definecolor{darkgoldenrod}{rgb}{0.72, 0.53, 0.04}
\definecolor{gold(metallic)}{rgb}{0.83, 0.69, 0.22}
\hypersetup{
 colorlinks=true,
 linkcolor=darkgoldenrod,
 filecolor=brown,      
 urlcolor=gold(metallic),
 citecolor=darkgoldenrod,
 }

\usepackage[all,cmtip]{xy}

\DeclareFontFamily{U}{wncy}{}
\DeclareFontShape{U}{wncy}{m}{n}{<->wncyr10}{}
\DeclareSymbolFont{mcy}{U}{wncy}{m}{n}
\DeclareMathSymbol{\Sh}{\mathord}{mcy}{"58}
\usepackage[T2A,T1]{fontenc}
\usepackage[OT2,T1]{fontenc}

\newtheorem{theorem}{Theorem}[section]
\newtheorem{lemma}[theorem]{Lemma}

\newtheorem*{theorem*}{Theorem}
\newtheorem*{ass*}{Assumption}
\newtheorem{definition}[theorem]{Definition}

\newtheorem{proposition}[theorem]{Proposition}

\newcommand{\cF}{\mathcal{F}}

\newcommand{\fl}{\mathfrak{l}}

\newcommand{\Z}{\mathbb{Z}}
\newcommand{\p}{\mathfrak{p}}
\newcommand{\Q}{\mathbb{Q}}
\newcommand{\F}{\mathbb{F}}

\newcommand{\cL}{\mathcal{L}}
\newcommand{\cC}{\mathcal{C}}
\newcommand{\cO}{\mathcal{O}}

\newcommand{\op}[1]{\operatorname{#1}}

\numberwithin{equation}{section}

\begin{document}

\title[Fine Selmer groups of Drinfeld modules]{Iwasawa theory of fine Selmer groups associated to Drinfeld modules}

\author[A.~Ray]{Anwesh Ray}
\address[Ray]{Chennai Mathematical Institute, H1, SIPCOT IT Park, Kelambakkam, Siruseri, Tamil Nadu 603103, India}
\email{anwesh@cmi.ac.in}


\maketitle

\begin{abstract}
 Let $q$ be a prime power and $F=\mathbb{F}_q(T)$ be the rational function field over $\mathbb{F}_q$, the field with $q$ elements. Let $\phi$ be a Drinfeld module over $F$ and $\p$ be a non-zero prime ideal of $A:=\mathbb{F}_q[T]$. Over the constant $\mathbb{Z}_p$-extension of $F$, we introduce the fine Selmer group associated to the $\mathfrak{p}$-primary torsion of $\phi$. We show that it is a cofinitely generated module over $A_{\mathfrak{p}}$. This proves an analogue of Iwasawa's $\mu=0$ conjecture in this setting, and provides context for the further study of the objects that have been introduced in this article.
\end{abstract}

\section{Introduction}
\par Let $K$ be a number field and $p$ be a prime number. Denote by $\Z_p$, the ring of $p$-adic integers. A $\Z_p$-extension $K_\infty/K$ is an infinite Galois extension of $K$ such that $\op{Gal}(K_\infty/K)$ is isomorphic to $\Z_p$ as a topological group. Let $K_n/K$ be the subfield of $K_\infty$ for which $\op{Gal}(K_n/K)\simeq \Z/p^n\Z$. Denote by $\op{Cl}_p(K_n)$ the $p$-primary part of the class group of $K_n$. We write $h_p(K_n)=p^{e_n}$ to denote the cardinality of $\op{Cl}_p(K_n)$. The fields $\{K_n\}$ form a $\Z_p$-tower of extensions 
\[K=K_0\subset K_1\subset \dots \subset K_n \subset K_{n+1}\subset \dots ,\]
Iwasawa showed that there are constants $\mu, \lambda\in \Z_{\geq 0}$ and $\nu\in \Z$ such that 
\[e_n=p^n \mu+ n \lambda +\nu\] for $n\gg 0$ (cf. \cite{iwasawa1973zl} or \cite[Chapter 13]{washington}). The quantities $\mu$ and $\lambda$ are the $\mu$ and $\lambda$-invariants associated to $K_\infty/K$ respectively. The \emph{cyclotomic $\Z_p$-extension} $K_{\op{cyc}}/K$ is the unique $\Z_p$-extension of $K$ which is contained in $K(\mu_{p^\infty})$. Iwasawa conjectured that for the cyclotomic $\Z_p$-extension of any number field $K$, the $\mu$-invariant vanishes. When $K/\Q$ is an abelian extension, the conjecture was resolved by Ferrero and Washington \cite{FW}.
\par Mazur (cf. \cite{mazur1972rational}) introduced a generalization of Iwasawa's theory in the context of Selmer groups associated to abelian varieties defined over a number field with good ordinary reduction at the primes above $p$. The Selmer groups are defined by local conditions prescribed the images of (local) Kummer maps. A closer analogy holds with the class group, when one considers the \emph{fine Selmer group}. This is the Selmer group for which the local conditions at primes of bad reduction are defined by the strict local condition. In this context, there is an analogue of the $\mu=0$ conjecture, that was formulated by Coates and Sujatha (cf. \cite[Conjecture A]{coates2005fine}). 

\par In this article, we study function field analogues of questions posed for elliptic curves over number fields. Drinfeld modules natural objects to consider over function fields that give rise to compatible systems of Galois representations. What is important is that for Drinfeld modules of \emph{generic characteristic} these representations cut out separable extensions and thus it is possible to define an analogue of the fine Selmer group in this context. Let $q$ be a prime power and $\F_q$ be the field with $q$ elements. Denote by $A$ the polynomial ring $\F_q[T]$ and $F$ its fraction field. Let $\phi$ be a Drinfeld module over $F$, and $\p$ be a non-zero prime ideal of $A$. Denote by $\cO$ the completion of $A$ at $\p$, and set $\mathcal{K}:=\op{Frac}\cO$. Then, its module of $\p$-primary torsion points $\phi[\p^\infty]$ is isomorphic to $(\mathcal{K}/\cO)^r$, where $r$ is the rank of $\phi$. 
\par Let $F^{\op{nr}}=\bar{\F}_q(T)$ be the maximal unramified extension of $F$, and choose isomorphisms 
    \[\op{Gal}(F^{\op{nr}}/F)\xrightarrow{\sim} \op{Gal}(\bar{\F}_q/\F_q)\xrightarrow{\sim} \widehat{\Z}.\] Denote by $\cF$ the unique $\Z_p$-extension of $F$ that is contained in $F^{\op{nr}}$. This $\Z_p$-extension is given by $\cF=\kappa (T)$, where $\kappa\subset \bar{\F}_q$ is a $\Z_p$-extension of $\F_q$. It is for this reason that $\cF$ is referred to as the \emph{constant $\Z_p$-extension} of $F$. It is of significance that any prime of $F$ is finitely decomposed in the infinite extension $\cF$. Let $\Gamma$ denote the Galois group $\Gamma:=\op{Gal}(\cF/F)$ and $\Lambda_{\cO}:=\cO\llbracket \Gamma \rrbracket$ be the \emph{Iwasawa algebra} over $\cO$. We define the fine Selmer group $\op{Sel}_{0}(\phi[\p^\infty]/\cF)$ to be the subgroup of $H^1(\op{Gal}(F^{\op{sep}}/\cF), \phi[\p^\infty])$ consisting of cohomology classes that are trivial when restricted to any prime of $\cF$. This Selmer group is a module over the Iwasawa algebra $\Lambda_{\cO}$ and is the precise analogue of the Selmer group considered in \cite{coates2005fine}. We study the algebraic structure of $\op{Sel}_{0}(\phi[\p^\infty]/\cF)$ and show that it is cofinitely generated and cotorsion over $\Lambda_{\cO}$. We define the Iwasawa $\mu$ and $\lambda$ invariants associated to $\op{Sel}_{0}(\phi[\p^\infty]/\cF)$, denoted $\mu_{\p}(\phi)$ and $\lambda_{\p}(\phi)$. We show that $\mu_{\p}(\phi)=0$ and that $\op{Sel}_{0}(\phi[\p^\infty]/\cF)$ is cofinitely generated as a $\cO$-module, and its $\cO$-corank is equal to $\lambda_{\p}(\phi)$. Moreover, we obtain an upper bound for the size of $\lambda_{\p}(\phi)$ in terms of the dimension of the \emph{residual fine Selmer group} $\op{Sel}_0(\phi[\p]/\cF)$ and certain local invariants associated to the primes at which $\phi$ has bad reduction and the primes $\p$ and $\infty$. Let $\F_{\p}$ denote the residue field of $\cO$. The main result is stated below.

    \begin{theorem}\label{MAIN thm}
    Let $\phi$ be a Drinfeld module over $F$ of rank $r$ and $\p$ be a non-zero prime of $A$. Denote by $S$ the set of primes $\{\p, \infty\}$ and the primes at which $\phi$ has bad reduction. Let $\cF$ be the constant $\Z_p$-extension of $F$. Then, the following assertions hold
    \begin{enumerate}
        \item \label{c1 MAIN thm} $\op{Sel}_0(\phi[\p^\infty]/\cF)$ is a cofinitely generated cotorsion $\Lambda_{\cO}$-module with $\mu_{\p}(\phi)=0$,
        \item \label{c2 MAIN thm} $\op{Sel}_0(\phi[\p^\infty]/\cF)$ is cofinitely generated as a module over $\cO$ and has corank equal to $\lambda_p(\phi)$.
        \item\label{c3 MAIN thm} Let $\op{Sel}_0(\phi[\p]/\cF)$ be the residual fine Selmer group (cf. Definition \ref{def of residual fine Selmer}). Then, we have the following bound for the $\lambda$-invariant
        \[\lambda_{\p}(\phi)\leq \op{dim}_{\F_\p}\op{Sel}_0(\phi[\p]/\cF)+\sum_{w\in S(\cF)} \dim_{\F_\p}\left(H^0(\cF_{w}, \phi[\p^\infty])\otimes_{\cO} \F_{\p}\right).\]
    \end{enumerate}
\end{theorem}
Thus, the equivalent of the $\mu=0$ conjecture of Coates and Sujatha is resolved for Drinfeld modules over $F$. In the bound for the $\lambda$-invariant above, the dimension of $\op{Sel}_0(\phi[\p]/\cF)$ is shown to be finite. The set of primes $S(\cF)$ is finite and the local invariants satisfy 
\[\dim_{\F_\p}\left(H^0(\cF_{w}, \phi[\p^\infty])\otimes_{\cO} \F_{\p}\right)\leq r.\] Furthermore, if $H^0(F_v, \phi[\p])=0$, then, $H^0(\cF_{w}, \phi[\p^\infty])=0$ for all primes $w|v$ (cf. Proposition \ref{H0 vanishing Prop}). Our technique leverages a theorem of Rosen, who proves $\mu=0$ for the $p$-primary class groups in a constant $\Z_p$-extension of a function field (cf. \cite[p.293, l.-2]{rosenmain} or \cite[Theorem 11.5]{rosentext}). The referee has pointed out that the methods in this paper should in principle generalize to the setting of Drinfeld modules over a function field, however, this has not been undertaken in the current work. 

\par The article consists of four sections, including the introduction. In section \ref{s 2} we discuss the theory of Drinfeld modules and set up relevant notation. In section \ref{s 3}, the fine Selmer group is introduced and we associate the Iwasawa invariants to this module. In the section \ref{s 4}, the main result is proven. We provide an illustrative example in subsection \ref{example section}.

\subsection*{Acknowledgment} The author thanks the anonymous referee for the excellent report.

\section{Drinfeld modules and associated Galois representations}\label{s 2}
\par In this section, we discuss preliminaries on Drinfeld modules. We refer the reader to \cite{papibook} for further details. Throughout, we shall let $q$ be a power of a prime number and $\F_q$ be the finite field with $q$ elements. We denote by $A$, the polynomial ring $\F_q[T]$, and $F:=\op{Frac}(A)=\F_q(T)$. A \emph{prime} of $F$ is an isomorphism class of valuations of $F$. Such a valuation gives rise to a discrete valuation ring contained in $F$, whose fraction field is $F$. Denote by $\Omega_F$ the set of all primes of $F$. A non-zero prime ideal $\mathfrak{l}$ contained in $A$ gives rise to an equivalence class of valuations on $F$. For ease of notation, we write $\fl\in \Omega_F$ to denote the associated equivalence class. Let $v_\infty$ denote the valuation at $\infty$, normalized by setting $v_\infty(T):=-1$. For ease of notation, $\infty$ shall simply be used in place of $v_\infty$. We set $\Omega_F':=\Omega_F\backslash \{\infty\}$; the set of primes that arise from non-zero prime ideals $\mathfrak{l}$ in $A$. Given any non-zero prime ideal $\mathfrak{a}$ in $A$, there is a unique monic polynomial $f_{\mathfrak{a}}$ that generates $A$. By abuse of notation simply denote $f_{\mathfrak{a}}$ by $\mathfrak{a}\in A$. Let $M$ be an $A$-module, set $M[\mathfrak{a}]$ to denote the kernel of the multiplication by $\mathfrak{a}$ map $M\xrightarrow{\mathfrak{a}} M$. 

\par For $\fl\in \Omega_F'$, set $\F_{\fl}$ to denote the residue field $A/\fl$ for $\fl\in \Omega_F'$. Given a finite set of primes $S\subset \Omega_F$, we let $F_S$ to denote the maximal separable extension of $F$ in which the primes $v\notin S$ are unramified. This is a Galois extension of $F$, we set $\op{G}_S:=\op{Gal}(F_S/F)$. A polynomial $f(x)$ is $\F_q$-linear if 
\[\begin{split}
    & f(x+y)=f(x)+f(y);\\
    & f(\alpha x)=\alpha f(x)\text{ for all }\alpha\in \F_q.
\end{split}\]Given an $\F_q$-algebra $K$, denote by $K\langle x\rangle$ be the ring of polynomials $f(x)\in K[x]$ that are $\F_q$-linear. For $f, g\in K\langle x\rangle$, we find that $f\circ g\in K\langle x\rangle$. This defines composition in the ring $K\langle x\rangle$. Let $K\{\tau\}$ be the ring of twisted polynomials 
\[f(\tau)=\sum_{i=0}^d a_i \tau^i\] with $a_i\in K$. In this ring, addition is given by 
\[\sum_i a_i \tau^i+\sum_i b_i \tau^i=\sum_i(a_i+b_i)\tau^i,\] and multiplication is prescribed by the relation
\[\begin{split}
    \tau^i \alpha= \alpha^{q^i} \tau^i.
\end{split}\]
Since $\alpha^q=\alpha$ for all $\alpha\in \F_q$, we find that $K\langle x\rangle $ is an $\F_q$-algebra.
\par There is a natural isomorphism of $\F_q$-algebras 
\[\eta: K\{\tau\}\xrightarrow{\sim} K\langle x\rangle,\] defined by 
\[\eta\left(\sum_{i=0}^d a_i \tau^i\right):=\sum_{i=0}^d a_i x^{q^i}. \]
For $f(\tau)\in K\{\tau\}$, we shall set $f(x):=\eta(f)$.
\begin{definition}
    Let $f(\tau)\in K\{\tau\}$, write 
    \[f(\tau)=\sum_{i=h}^d a_i \tau^i, \] where $a_h a_d\neq 0$. Then, the \emph{height} (resp. \emph{degree}) of $f$ is defined by setting $\op{ht}_\tau(f):=h$ (resp. $\op{deg}_\tau (f):=d$). The formal derivative of $f(x)$ is given by $\partial (f)=a_0$ (taken to be $0$ if $\op{ht}_\tau(f)>0$). We find that $f(x)$ is separable if an only if $\op{ht}_\tau(f)=0$. 
\end{definition}
Let $K$ be an $\F_q$-algebra, we say that $K$ is an $A$-field if there is a prescribed homomorphism of $\F_q$-algebras $\gamma: A\rightarrow K$. 
\begin{definition}
    Let $K$ be an $A$-field and $r\in \Z_{\geq 1}$. A \emph{Drinfeld module} of rank $r$ over $K$ is a homomorphism of $\F_q$-algebras
    \[\phi: A\rightarrow K\{\tau\},\]
    taking $a\in A$ to $\phi_a\in K\{\tau\}$, such that \begin{enumerate}
        \item $\partial(\phi_a)=\gamma(a)$,
        \item $\op{deg}_\tau(\phi_a)=r \op{deg}_T(a)$.
    \end{enumerate}
\end{definition}
We shall set $r(\phi)$ to denote the rank of $\phi$. Since $\phi$ is a homorphism of $\F_q$-algebras, it completely determined by 
\[\phi_T=\gamma(T)+\sum_{i=1}^r a_i \tau^i,\] where $a_r\neq 0$. 
\begin{definition}\label{defn drinfeld}
    Let $(K, \gamma)$ be an $A$-field. We set $\op{char}_A(K):=0$ if $\gamma$ is injective. We say in this case that $K$ is \emph{generic characteristic}. On the other hand, if $\mathcal{P}:=\op{ker}\gamma$ is a non-zero prime ideal of $A$, we set $\op{char}_A(K):=\mathcal{P}$.
\end{definition} We note that if $K$ has generic characteristic, then $\partial(\phi_a)=\gamma(a)\neq 0$ for all $a\neq 0$ in $A$. Therefore, in this case, $\phi_a(x)$ is a separable polynomial for all $a\in A$. When $\op{char}_A(K)=\mathcal{P}\neq 0$, we find that $\phi_a(x)$ is separable polynomial if and only if $\mathcal{P}\nmid a$. Given a non-zero ideal $\mathfrak{a}$ of $A$, define $\phi_{\mathfrak{a}}:=\phi_a$, where $a$ is the monic polynomial that generates $\mathfrak{a}$. The degree $\op{deg}(\mathfrak{a})$ is defined to be $\op{deg}_T(a)$, the degree of $a$ as a polynomial in $T$. It follows from the Definition \ref{defn drinfeld} that $\op{deg}_\tau(\phi_a)=r \op{deg}(\mathfrak{a})$. 

\par Since $\mathcal{P}$ is the kernel of $\gamma$ and $\partial \phi_a=\gamma(a)$, the height of $\phi_{\mathcal{P}}$ is positive. Moreover, it is divisible by $\op{deg}(\mathcal{P})$ (cf. \cite[Lemma 3.2.11]{papibook}). The \emph{height} of $\phi$ is defined as follows
\[H(\phi):=\begin{cases}
 0 & \text{ if }\op{char}_A(K)=0;\\
 \left(\frac{\op{ht}_\tau(\phi_\mathcal{P})}{\op{deg}\mathcal{P}}\right)& \text{ if }\op{char}_A(K)=\mathcal{P}\neq 0.\\ \end{cases}\]
 Given a Drinfeld modules $\phi$ and $\psi$ over an $A$-field $K$, a \emph{morphism} $u\in \op{Hom}_K(\phi, \psi)$ is an element $u\in K\{\tau\}$ such that the relationship $u \phi_a=\psi_a u$ holds for all $a\in A$. The endomorphism ring of $\phi$ (over $K$) is defined as follows
\[\op{End}_K(\phi):=\op{Hom}_K(\phi, \phi)=\{u\in K\{\tau\}\mid u \phi_a=\phi_a u\text{ for all }a\in A\}.\] This ring is in fact is an $A$-algebra, with $a\cdot u:=\phi_a u=u\phi_a$. Denote by 
\[\op{str}_\phi=\op{str}_{\phi, K}: A\rightarrow \op{End}_K(\phi)\] the \emph{structure homomorphism}, mapping $a$ to $\phi_a$. Since $\op{deg}_\tau(\phi_a)=r\op{deg}_T(a)$, it follows that $\op{str}_{\phi, K}$ is injective. We shall write $\op{End}_K(\phi)=A$ to mean that this structure homomorphism is an isomorphism.
\par Let $\phi$ and $\psi$ be isomorphic Drinfeld modules over $K$ and suppose $u\in \op{Hom}(\phi, \psi)$ is an isomorphism. Then, \[\op{deg} u+\op{deg}(u^{-1})=\op{deg}_{\tau}(u u^{-1})=\op{deg}_{\tau}(1)=0.\] Therefore, $u\in K$ is a constant, we find that $u\phi u^{-1}=\psi$. Writing 
\[\phi_T=\sum_{i=0}^r g_i \tau^i,\]
and $u=c^{-1}$, we find that 
\[\begin{split}\psi_T=& \sum_{i=0}^r c^{-1} g_i \tau^i c \\ 
=& \sum_{i=0}^r c^{q^i-1} g_i \tau^i .\end{split}\]
\par Let $\bar{K}$ (resp. $K^{\op{sep}}$) be an algebraic closure (resp. separable closure) of $K$. We choose these fields so that $K^{\op{sep}}$ is contained in $\bar{K}$. We set $\op{G}_K$ to denote the absolute Galois group $\op{Gal}(K^{\op{sep}}/K)$. Let $\phi:A\rightarrow K\{\tau\}$ be a Drinfeld module, then we have a $\phi$-twisted $A$-module structure on $\bar{K}$, defined as follows. Let $y\in \bar{K}$ and $a\in A$. Then, set 
\[a\cdot y:=\phi_a(y).\]Let $E/K$ be an extension that is contained in $\bar{K}$. For $y\in E$, we have that $\phi_a(y)\in E$. We denote by $^\phi E$ denote $E$ with this twisted $A$-module structure. Let $a\in A$ be a non-constant element. Assume that if $\mathfrak{p}=\op{char}_A(K)\neq 0$, then, $\mathfrak{p}\nmid a$. Denote by $\phi[a]\subset \bar{K}$ the set of roots of $\phi_a(x)$. It is easy to see that $\phi[a]$ is an $A$-submodule of $^{\phi}\bar{K}$. Moreover, it is easy to see that it is stable with respect to the action of $\op{G}_K$, and that $\op{G}_K$ acts on $\phi[a]$ by $A$-module automorphisms. As an $A$-module $\phi[a]\simeq \left(A/a\right)^r$, cf. \cite[Theorem 3.5.2]{papibook}. We choose an $A/a$-module basis of $\phi[a]$, the associated Galois representation is denoted
\[\rho_{\phi, a}: \op{G}_K\rightarrow \op{GL}_r\left(A/a\right).\] Given a non-zero ideal $\mathfrak{a}$ of $A$, set $\phi[\mathfrak{a}]:=\phi[b]$, where $b$ is the monic polynomial generator of $\mathfrak{a}$. Given a non-zero prime ideal $\fl\neq \mathcal{P}$ of $A$, set 
\[\phi[\fl^\infty]:=\bigcup_{i\geq 1} \phi[\fl^i].\]
Denote by $A_{\fl}:=\varprojlim_i A/\fl^i$ denote the completion of $A$ at $\fl$, and set $F_{\fl}:=\op{Frac}(A_{\fl})$. We an isomorphism of $A$-modules $\phi[\fl^\infty]\simeq \left(F_{\fl}/A_{\fl}\right)^r$. The Tate-module $T_{\fl}(\phi)$ the inverse limit 
\[T_{\fl}(\phi):=\varprojlim_n \phi[\fl^n], \] with respect to multiplication by $\phi_{\fl}$-maps. We choose compatible isomorphisms $\phi[\fl^n]\simeq \left(A/\fl^n\right)^r$, and thus find that $T_{\fl}(\phi)\simeq A_{\fl}^r$. The Galois group $\op{G}_K$ acts on $T_{\fl}(\phi)$ by $A_{\fl}$-linear automorphisms and let 
\[\rho_{\phi, \fl^\infty}: \op{G}_K\rightarrow \op{GL}_r(A_{\fl})\] be the associated Galois representation. 

\par Let $\phi$ be a Drinfeld module of rank $r$ over $F$, and $\p$ a non-zero prime ideal of $A$. Denote by 
\[\rho: \op{G}_F\rightarrow \op{GL}_r(A_{\p})\] the Galois representation associated to $T_{\p}(\phi)$. For $\fl\in \Omega_F'$, let $\phi_{\fl}$ be the base change of $\phi$ to $F_{\fl}$. 
\begin{definition}
    With respect to notation above, we say that $\phi$ has \emph{stable reduction} at $\fl$ if it is isomorphic (over $F_{\fl}$) to a Drinfeld module 
    \[\psi:A\rightarrow A_{\fl}\{\tau\}\]
    whose reduction $\bar{\psi}$ is a Drinfeld module. We say that $\psi$ has good reduction if the rank of $\bar{\psi}$ is equal to $r$ (the rank of $\phi$). We say that $\phi$ has bad reduction at $\fl$ if it does not have good reduction at $\fl$.
\end{definition}
Let $S_{\op{bad}}(\phi)$ be the set of primes $\fl\in \Omega_F'$ at which $\phi$ has bad reduction. We write 
\[\phi_T=T+a_1\tau+a_2\tau^2+\dots+a_r \tau^r.\] After replacing $\phi$ by an isomorphic Drinfeld module, we assume, without loss of generality that all the coefficients $a_i\in A$. The set of primes $S_{\op{bad}}(\phi)$ is a subset of the prime divisors of $a_r$, and therefore, is a finite set. The representation $\rho$ is unramified at all primes $\fl\in \Omega_F$ such that $\fl\notin S_{\op{bad}}(\phi)\cup \{\p, \infty\}$, cf. \cite[Theorem 6.3.1]{papibook}.

\section{Fine Selmer groups and their properties}\label{s 3}
\par Recall that $F$ is the fraction field $\F_q(T)$. The map $\gamma: A\hookrightarrow F$ is the natural inclusion, and $F$ is an $A$-field with respect to this choice of $\gamma$. Let $\phi$ be a Drinfeld module over $F$ of rank $r$. Let $\p$ be a fixed non-zero prime ideal of $A$. Then, $\phi[\p^\infty]$ is unramified at all non-zero prime ideals $\fl\neq \p$ at which $\phi$ has good reduction. We shall denote by $\Omega_F$ to denote the set of primes of $F$. These consist of the non-zero prime ideals of $A$, and the prime $\infty$, that corresponds to the valuation $v_\infty$ normalized by $v_\infty(T):=-1$. Let $S\subset \Omega_F$ be a finite set of primes containing $\{\p, \infty\}$ and $S_{\op{bad}}(\phi)$. Since $F$ has generic characteristic, $\phi[\p^\infty]$ is contained in $F^{\op{sep}}$. Let $L$ be a separable Galois extension of $F$ contained in $F_S$. Given a prime $v\in \Omega_F$, set $v(L)$ to denote the set of primes of $L$ that lie above $v$. For $w\in v(L)$, set $L_w$ to denote the union of completions $L'_w$, as $L'$ ranges over all finite extensions of $F$ that are contained in $L$. Take $H^i(L_w, \cdot)$ to denote $H^i(L_w^{\op{sep}}/L_w, \cdot)$. Given a finite set of primes $\Sigma$ of $F$, we set $\Sigma(L)$ to be the set of primes $w$ that lie above a prime $v\in \Sigma$. 

\par We define the local condition at $v$ as follows
\[J_v(\phi[\p^\infty]/L):=\prod_{w\in v(L)} H^1\left(L_{w}, \phi[\p^\infty]\right).\]
\begin{definition}
    The \emph{fine Selmer group} of $\phi[\p^\infty]$ is defined as follows
    \[\op{Sel}_0^S(\phi[\p^\infty]/L):=\op{ker}\left(H^1(F_S/L, \phi[\p^\infty])\longrightarrow \bigoplus_{v\in S} J_v(\phi[\p^\infty]/L)\right).\]
\end{definition}
Let $G:=\op{Gal}(L/K)$, denote by $\cO$ the valuation ring $A_{\p}$, and take $\mathcal{K}:=\op{Frac}\cO$. Let $\varpi$ be a uniformizer of $\cO$. The fine Selmer group defined above is an $\cO$-module, and also a module over $G$. The action of $G$ on $\op{Sel}_0^S(\phi[\p^\infty]/L)$ is via $\cO$-module automorphisms. The fine Selmer group $\op{Sel}_0^S(\phi[\p^\infty]/L)$ can be identified with the direct limit
\[\op{Sel}_0^S(\phi[\p^\infty]/L)=\varinjlim_{L'}  \op{Sel}_0^S(\phi[\p^\infty]/L'),\] where $L'/K$ runs over all finite Galois extensions of $K$ that are contained in $L$. 
\par The \emph{Iwasawa algebra} of $G$ over $\cO$ is the completed group algebra 
\[\cO\llbracket G\rrbracket:= \varprojlim_N \cO[G/N],\] where $N$ runs over all finite index normal subgroups of $G$. For each finite extension $L'/K$, we find that $\op{Sel}_0^S(\phi[\p^\infty]/L')$ is an $\cO[\op{Gal}(L'/K)]$-module. The Pontryagin-dual of $\op{Sel}_0^S(\phi[\p^\infty]/L)$ is denoted
\[Y^S(\phi[\p^\infty]/L):=\op{Hom}_{\cO-\op{mod}}\left(\op{Sel}_0^S(\phi[\p^\infty]/L), \mathcal{K}/\cO\right),\] and is a module over $\cO\llbracket G\rrbracket$. Let $F^{\op{nr}}=\bar{\F}_q(T)$ be the maximal unramified extension of $F$, and choose isomorphisms 
    \[\op{Gal}(F^{\op{nr}}/F)\xrightarrow{\sim} \op{Gal}(\bar{\F}_q/\F_q)\xrightarrow{\sim} \widehat{\Z}.\] Denote by $\cF$ the unique $\Z_p$-extension of $F$ that is contained in $F^{\op{nr}}$. This $\Z_p$-extension is given by $\cF=\kappa (T)$, where $\kappa\subset \bar{\F}_q$ is a $\Z_p$-extension of $\F_q$. It is for this reason that $\cF$ is referred to as the \emph{constant $\Z_p$-extension} of $F$. Set $\Gamma:=\op{Gal}(\cF/F)$ and $\Lambda_{\cO}=\cO\llbracket \Gamma \rrbracket$. The fine Selmer group $\op{Sel}_0^S(\phi[\p^\infty]/\cF)$ and its Pontryagin dual $Y^S(\phi[\p^\infty]/\cF)$ are naturally modules over $\Lambda_{\cO}$. There is a natural Galois action of $\Gamma$ on these modules, that arises from the canonical inclusion $\Gamma\hookrightarrow \Lambda_{\cO}^\times$. In greater detail, for $\phi\in \op{Sel}_0^S(\phi[\p^\infty]/\cF)$, $\psi\in Y^S(\phi[\p^\infty]/\cF)$ and $g\in \Gamma$, one has that $(g\cdot \psi)(\phi)=\psi(g^{-1}\cdot \phi)$.
\begin{proposition}\label{fine Selmer indep}
    Let $S$ and $S'$ be finite sets of primes containing $S_{\op{bad}}(\phi)$ and $\{\p, \infty\}$. Then, we have that 
    \[\op{Sel}_0^S(\phi[\p^\infty]/\cF)=\op{Sel}_0^{S'}(\phi[\p^\infty]/\cF).\]
\end{proposition}
\begin{proof}
  We assume without loss of generality that $S\subseteq S'$, set $S'':=S'\backslash S$. For $w\in S''$, let $\kappa_w$ denote the residue field at $w$. Denote by $\op{I}_w$ the inertia subgroup of $\op{Gal}(\bar{\cF}_w/\cF_w)$. Since $\phi[\p^\infty]$ is unramified at $w$, we find that $\phi[\p^\infty]^{\op{I}_w}=\phi[\p^\infty]$.
  \par Set $H^1_{\op{nr}}(\cF_w, \phi[\p^\infty])$ denote the image of the inflation map 
  \[\op{inf}: H^1(\kappa_w, \phi[\p^\infty])\hookrightarrow H^1(\cF_w, \phi[\p^\infty]).\] We identify $H^1_{\op{nr}}(\cF_w, \phi[\p^\infty])$ with $H^1\left(\op{Gal}(\bar{\kappa}_w/\kappa_w),\phi[\p^\infty]\right)$. Observe that $\op{Sel}_0^{S'}(\phi[\p^\infty]/\cF)$ consists of classes in $\op{Sel}_0^{S}(\phi[\p^\infty]/\cF)$ that are unramified at all primes $w\in S''(\cF)$. In other words, we have a left exact sequence
    \begin{equation}\label{ses1}0\rightarrow \op{Sel}_0^{S'}(\phi[\p^\infty]/\cF)\rightarrow \op{Sel}_0^{S}(\phi[\p^\infty]/\cF)\rightarrow \bigoplus_{w\in S''(\cF)} H^1\left(\op{Gal}(\bar{\kappa}_w/\kappa_w), \phi[\p^\infty]\right).\end{equation}
   Since all primes of $F$ are finitely decomposed in $\cF$, we find that $\op{Gal}(\kappa_w/\F_q)\simeq \Z_p$, and hence, $\op{Gal}(\bar{\kappa}_w/\kappa_w)\simeq \prod_{\ell\neq p} \Z_\ell$. On the other hand, $\phi[\p^\infty]$ is an $\F_p$-vector space and $\op{Gal}(\bar{\kappa}_w/\kappa_w)$ has no $p$-primary subquotients. It follows that  \[H^1\left(\op{Gal}(\bar{\kappa}_w/\kappa_w), \phi[\p^\infty]\right)=0\]
   for all primes $w$. Thus, it follows from \eqref{ses1} that 
    \[\op{Sel}_0^S(\phi[\p^\infty]/\cF)=\op{Sel}_0^{S'}(\phi[\p^\infty]/\cF).\]
\end{proof}
Proposition \ref{fine Selmer indep} implies that the fine Selmer group over $\cF$ is independent of the choice of primes $S$ containing $\{\p, \infty\}$ and the primes at which $\phi$ has bad reduction. Henceforth, we suppress the dependence on $S$ by simply writing $\op{Sel}_0(\phi[\p^\infty]/\cF)$ to denote this fine Selmer group. We recall that $Y(\phi[\p^\infty]/ \cF)$ is a module over $\Lambda_{\cO}:=\cO\llbracket \Gamma\rrbracket$. Choose a topological generator $\gamma$ of $\Gamma$. Setting $T:=(\gamma-1)$, we identify $\Lambda_{\cO}$ with the formal power series ring $\cO\llbracket T\rrbracket$. We shall study the algebraic structure of this module. 
\par We recall that a polynomial $f(T)\in \cO[T]$ is \emph{distinguished} if it is a monic polynomial whose non-leading coefficients all belong to $(\varpi)$. The Weierstrass preparation theorem states that any non-zero element $f(T)\in \Lambda$ can be uniquely factored as 
\[f(T)=\varpi^n u(T) g(T),\]
where $u(T)$ is a unit in $\Lambda$ and $g(T)$ is a distinguished polynomial, cf. \cite[Theorem 2.7.10]{papibook}.

\begin{proposition}\label{structure thm}
    Let $M$ be a finitely generated $\cO\llbracket T\rrbracket$-module. By the structure theory of finitely generated and torsion modules over $\cO\llbracket T\rrbracket$, there is a homomorphism 
\[M\longrightarrow \cO\llbracket T\rrbracket^r\oplus \left(\bigoplus_{i=1}^s \frac{\cO\llbracket T\rrbracket}{(\varpi^{\mu_i})}\right)\oplus \left( \bigoplus_{i=1}^s \frac{\cO\llbracket T\rrbracket}{(f_j(T))} \right),\] whose kernel and cokernel are both finite. In the above decomposition, $f_j(T)$ is a distinguished polynomial. The module $M$ is torsion over $\cO\llbracket T\rrbracket$ if and only if $r=0$. 
\end{proposition}
\begin{proof}
    The result is proven \cite[Ch. 13.1]{washington} for Iwasawa algebras over $\Z_p$. This proof generalizes verbatim to Iwasawa algebras over a general discrete valuation ring (in all the arguments in \emph{loc. cit.}, replace $p$ with $\varpi$).
\end{proof}

Let $M$ be a finitely generated and torsion $\Lambda$-module. The \emph{characteristic element} is defined as follows
\[f_M(T):=\prod_i \varpi^{\mu_i} \times \prod_j f_j(T).\] This element decomposes into a product 
\[f_M(T)=\varpi^\mu\times g(T),\] where $\mu\in \Z_{\geq 0}$ and $g(T)$ is a distinguished polynomial. The $\mu$-invariant is the quantity $\mu=\sum_i \mu_i$ in the above decomposition. The $\lambda$-invariant is the degree of $g(T)$. A module $M$ over $\cO$ is said to be $\p$-primary if $M=\bigcup_{i\geq 1} M[\p^i]$. Given a $\p$-primary $\cO$-module, the Pontryagin dual is 
\[M^\vee=\op{Hom}_{\cO-\op{mod}}\left(M, \mathcal{K}/\cO\right).\] A primary $\Lambda_{\cO}$-module is said to be cofinitely generated (resp. cotorsion) is $M^\vee$ is a finitely generated (resp. torsion) as a $\Lambda_{\cO}$-module. 

\begin{lemma}\label{finitely generated N lemma}
    Let $N$ be a compact $\Lambda_{\cO}$-module. Then, $N$ is finitely generated if and only if $N/(T,\varpi) N$ is finite dimensional as an $\F_q$ vector space. 
\end{lemma}
\begin{proof}
    The proof of the result above is identical to that \cite[Lemma 13.16]{washington}, upon replacing $p$ with $\varpi$ in \emph{loc. cit}.
\end{proof}
Let $M$ and $M'$ be modules over $\Lambda_{\cO}$. A \emph{pseudo-isomorphism} $f:M\rightarrow M'$ is a homomorphism of $\Lambda_{\cO}$-modules whose kernel and cokernel are both finite. We say that $M$ is pseudo-isomorphic to $M'$ if there is a pseudo-isomorphism $f:M\rightarrow M'$. 

\begin{lemma}\label{main lemma 1}
    Let $M$ be a $\p$-primary $\Lambda_{\cO}$-module and let $\mu_{\p}(M)$ (resp. $\lambda_{\p}(M)$) denote the associated $\mu$ (resp. $\lambda$) invariant. Then, the following are equivalent
    \begin{enumerate}
        \item\label{c1 main lemma 1} $M$ is a cofinitely generated, cotorsion as a $\Lambda_{\cO}$-module with $\mu_{\p}(M)=0$,
        \item\label{c2 main lemma 1} $M[\varpi]$ is finite.
    \end{enumerate}
    Moreover, if the above conditions are satisfied, then, $\lambda\leq \op{dim}_\kappa(M[\varpi])$.  
\end{lemma}

\begin{proof}
    Assume that $M$ is is a cofinitely generated, cotorsion as a $\Lambda_{\cO}$-module with $\mu_{\p}(M)=0$, and set $N:=M^\vee$. We note that $F_\p/A_\p$ is a divisible $A_\p$-module, and thus is injective. Moreover, for all $n\geq 1$, there is a natural isomorphism of $N/(\varpi^n) N$ with $\left(M[\varpi^n]\right)^\vee$. Then by Proposition \ref{structure thm}, there is a pseudo-isomorphism $f:N\rightarrow N'$, where  \[N'\xrightarrow{\sim}\left( \bigoplus_{i=1}^s \frac{\cO\llbracket T\rrbracket}{(f_j(T))} \right).\] Here, $f_j(T)$ are distinguished polynomials, and $N'$ is a free $\cO$-module of rank \[\lambda=\sum_j \op{deg} f_j(T).\] Thus, $N'/\varpi N'$ is finite. Since $N$ is pseudo-isomorphic to $N'$, it follows that $N/\varpi N$ is finite, and hence, $M[\varpi]$ is finite.
    \par Conversely, suppose that $M[\varpi]$ is finite. Then from the exact sequence
    \[0\rightarrow M[\varpi]\rightarrow M[\varpi^{n+1}]\rightarrow M[\varpi^n],\] it follows that $M[\varpi^n]$ is finite for all $n$. Let $N$ be the Pontryagin dual of $M$, we find that $N/\varpi^n$ is finite for all $n$. Since $M$ is a $\p$-primary $\Lambda_{\cO}$-module, we have that $M=\varinjlim_n M[\varpi^n]$. Therefore, $N=\varprojlim_n N/\varpi^n$ is a compact $\Lambda_{\cO}$-module. Since $N/\varpi N$ is finite, it follows from Lemma \ref{finitely generated N lemma} that $N$ is finitely generated as a $\Lambda_{\cO}$-module. It then follows from Proposition \ref{structure thm} that $N$ is a torsion $\Lambda_{\cO}$-module with $\mu=0$ if and only if it is pseudo-isomorphic to $N'$, where $N'$ is a free $\cO$-module of finite rank. In particular, this implies that $N/\varpi N$ is finite. Since $M[\varpi]$ is dual to $N/\varpi N$, we deduce that $M[\varpi]$ is finite. 

    \par Assume that the equivalent conditions are satisfied. It follows from Proposition \ref{structure thm} that $N$ is pseudo-isomorphic to a free $\cO$-module of rank $\lambda$. Thus, as a $\cO$-module, $N$ is finitely generated of rank equal to $\lambda$. Thus, we find that \[\dim_{\F_q} M[\varpi]=\dim_{\F_q} N/\varpi N\geq \lambda. \]
\end{proof}

\begin{definition}Suppose for the sake of discussion that the dual fine Selmer group $Y(\phi[\p^\infty]/\cF)$ is a finitely generated and torsion $\cO\llbracket \Gamma \rrbracket$-module. We shall show in section \ref{s 4} that this is the case. The $\mu$ and $\lambda$ invariants of $Y(\phi[\p^\infty]/\cF)$ are then denoted $\mu_{\p}(\phi)$ and $\lambda_{\p}(\phi)$ respectively. 
\end{definition}
\section{Main results}\label{s 4}
\subsection{Iwasawa theory of class groups over function fields}
\par In this section, let $L$ be a function field (of transcendence degree $1$) over $\F_q$. Denote by $\Omega_L$ the set of all primes of $L$. Given $v\in \Omega_L$, set $L_v$ to denote the completion of $L$ at $v$. The ring of integers of $L_v$ is denoted $\cO_{L_v}$. We let $\mathbb{A}_L$ be the ring of adeles of $L$, and $\cC_L:=\mathbb{A}_L^\times/L^\times$ the idele class group. The Artin map gives an isomorphism 
\[\op{Art}:\widehat{\cC_L}\xrightarrow{\sim} \op{Gal}(L^{\op{ab}}/L).\] This allows one to parametrize all finite separable abelian extensions of $L$. A Weil divisor $D$ is a finite sum $\sum_i n_i v_i$, where $v_i\in \Omega_L$ and $n_i\in \Z$. The Weil divisor associated to a non-zero element $a\in L^\times$ is denoted $(a)$. Two Weil divisors $D$ and $D'$ are equivalent if $D-D'$ is a principal Weil divisor, i.e., of the form $(a)$ for some $a\in L^\times$. The class group of $L$, denoted by $\op{Cl}(L)$ is the group of equivalence classes of Weil divisors. We set $\op{Cl}_p(L)$ to denote the $p$-primary part of $\op{Cl}(L)$.

Let $S$ be a non-empty finite set of primes of $L$ and $\cO_S$ denote the ring of $S$-integers in $L$. Denote by $\op{Cl}_p(\cO_S)$ the $p$-primary part of the divisor class group of $\op{Spec}(\cO_S)$. This is identified with the quotient of $\op{Cl}_p(L)$ by the divisor classes generated by all primes $v\in S$.  There is a natural isomorphism $\op{Cl}_p(\cO_S)\simeq \op{Gal}(H_p^S(L)/L)$, where $H_p^S(L)$ is the maximal $p$-primary abelian extension of $L$ contained in $L^{\op{sep}}$ that is unramified at all primes of $L$ and in which all primes of $S$ are split (cf. \cite[p.64, l.-3]{H-Koch}). There is a natural surjection $\op{Cl}(L)\rightarrow \op{Cl}(\cO_S)$. We set $h_p^S(L)$ to denote the number of elements in $\op{Cl}_p(\cO_S)$. Let $L_n$ be the unramified $p^n$-extension of $L$, and $\cL:=\bigcup_n L_n$ the constant $\Z_p$-extension of $L$. Set $H_n:=H_p^S(L_n)$, and $X_n:=\op{Gal}(H_n/L_n)$. Denote by $H_\infty$ the union $\bigcup_n H_n$, and $X:=\op{Gal}(H_\infty/\mathcal{L})$. Setting $\Lambda_{\Z_p}:=\Z_p\llbracket \op{Gal}(\cL/L)\rrbracket\simeq \Z_p\llbracket T \rrbracket$, we observe that $X$ is a $\Lambda$-module. Set $\mu_p(L)$ (resp. $\lambda_p(L)$) denote the $\mu$ and $\lambda$-invariants associated to $X$, as a module over $\Lambda_{\Z_p}$. We refer to \cite{lizhao} for a survey on the algebraic Iwasawa theory of class groups for global function fields.
\begin{theorem}\label{ROSEN thm}
    The $\mu$-invariant $\mu_p^S(L)=0$.
\end{theorem}
\begin{proof}
    The vanishing of the $\mu$-invariant follows from results of Rosen (cf. \cite[p.293, l.-2]{rosenmain}).
\end{proof}
\subsection{The vanishing of the $\mu$-invariant of the fine Selmer group}
\par At the end of this subsection, we give a proof of Theorem \ref{MAIN thm}. We define the fine Selmer group associated to the residual representation. Let $S$ be a finite set of primes containing $\{\p, \infty\}$ and the primes at which $\phi$ has bad reduction. For $v\in S$, set 
\[J_v(\phi[\p]/\cF):=\prod_{w\in v(\cF)} H^1\left(\cF_{w}, \phi[\p]\right).\]
\begin{definition}\label{def of residual fine Selmer}
    With respect to the above notation, the \emph{residual fine Selmer group} is defined as follows
\[\op{Sel}_0(\phi[\p]/\cF):=\op{ker}\left(H^1(F_S/\cF, \phi[\p])\xrightarrow{\bar{\Phi}} \bigoplus_{v\in S} J_v(\phi[\p]/\cF)\right),\] where $\bar{\Phi}$ is the product of restriction maps.
\end{definition}
By an argument identical to that in the proof of Proposition \ref{fine Selmer indep}, $\op{Sel}_0(\phi[\p]/\cF)$ is independent of the choice of the set of primes $S$. 

\par Let $k$ be a Galois extension of $F$ contained in $F_S$ or a separable extension of $F_{\fl}$ for some prime $\fl$. Denote by $\F_{\p}$ the residue field $A/\p$. The Kummer sequence 
\[0\rightarrow \phi[\p]\rightarrow \phi [\p^\infty]\xrightarrow{\phi_{\p}} \phi[\p^\infty]\rightarrow 0\] is a short exact sequence of $\op{G}_k$-modules, associated to which we have a short exact sequence
\begin{equation}\label{kummer}0\rightarrow H^0(k, \phi[\p^\infty])\otimes_{\cO} \F_{\p}\rightarrow H^1(k, \phi[\p])\xrightarrow{\beta_k} H^1(k, \phi[\p^\infty])[\p]\rightarrow 0.\end{equation}Given a prime $v\in S$, and $w\in v(\cF)$, let 
\[h_w:H^1(\cF_{ w}, \phi[\p])\rightarrow H^1(\cF_{w}, \phi[\p^\infty])[\p]\]denote the natural map in the Kummer sequence \eqref{kummer}. Set
\[h_v: J_v( \phi[\p]/\cF)\rightarrow J_v(\phi[\p^\infty]/\cF)[\p]\] to be the product of maps
\[\prod_{w\in v(\cF)} h_w: \prod_{w\in v(\cF)} H^1(\cF_{ w}, \phi[\p])\rightarrow \prod_{w\in v(\cF)} H^1(\cF_{w}, \phi[\p^\infty])[\p].\] Denote by $h$ the direct sum of the maps $h_v$, as $v$ ranges over $S$
\[h:\bigoplus_{v\in S}J_v(\phi[\p]/\cF) \longrightarrow \bigoplus_{v\in S}J_v( \phi[\p^\infty]/\cF)[\p].\]
On the other hand, the map 
\[\beta=\beta_F:H^1(F_S/\cF, \phi[\p])\rightarrow H^1(F_S/\cF, \phi[\p^\infty])\] restricts to a natural map 
\[\gamma: \op{Sel}_0(\phi[\p]/\cF)\xrightarrow{\gamma} \op{Sel}_0(\phi[\p^\infty]/\cF)[\p],\] that fits into a commutative diagram
\begin{equation}\label{fdiagram}
\begin{tikzcd}[column sep = small, row sep = large]
0\arrow{r} & \op{Sel}_0(\phi[\p]/\cF) \arrow{r}\arrow{d}{\gamma} & H^1(F_S/\cF, \phi[\p])\arrow{r} \arrow{d}{\beta} & \operatorname{im}(\bar{\Phi})\arrow{r} \arrow{d}{h'} & 0\\
0\arrow{r} & \op{Sel}_0(\phi[\p^\infty]/\cF)[\p]\arrow{r} & H^1(F_S/\cF, \phi[\p^{\infty}])[\p] \arrow{r}  &\bigoplus_{v\in S} J_v(\phi[\p^{\infty}]/\cF)[p].
\end{tikzcd}
\end{equation}
In the above diagram $h'$ is the restriction of $h$ to $\op{im}(\bar{\Phi})$.

\begin{lemma}\label{lemma 4.3}
    The natural map
    \[\op{Sel}_0(\phi[\p]/\cF)\xrightarrow{\gamma} \op{Sel}_0(\phi[\p^\infty]/\cF)[\p]\]has finite kernel and cokernel, whose size satisfy the upper bounds
    \[\begin{split}&  \# \op{ker}\gamma \leq  \# \left(H^0(\cF, \phi[\p^\infty])\otimes_{\cO}\F_{\p}\right)\leq |\F_\p|^r, \\ & \# \op{cok}\gamma\leq \prod_{w\in S(\cF)} \# \left(H^0(\cF_{w}, \phi[\p^\infty])\otimes_{\cO} \F_\p\right)\leq |\F_\p|^{\#S(\cF)r}.\end{split}\]
\end{lemma}
\begin{proof}
    From the commutative diagram \eqref{fdiagram}, we obtain the exact sequence 
    \[0\rightarrow \op{ker}\gamma \rightarrow \op{ker}\beta\rightarrow \op{ker}h'\rightarrow \op{cok}\gamma \rightarrow 0.\]
    Therefore, in particular, we find that 
    \[\begin{split}&  \# \op{ker}\gamma \leq \#\op{ker}\beta= \# \left(H^0(\cF, \phi[\p^\infty])\otimes_{\cO}\F_{\p}\right), \\ & \# \op{cok}\gamma\leq \#\op{ker}h'\leq \# \op{ker} h\leq \prod_{w\in S(\cF)} \# \left(H^0(\cF_{w}, \phi[\p^\infty])\otimes_{\cO} \F_\p\right).\end{split}\]
    Since $\phi[\p^\infty]\simeq \left(\mathcal{K}/\cO\right)^r$, we find that 
    \[\# \left(H^0(\mathcal{F}, \phi[\p^\infty])\otimes_{\cO}\F_{\p}\right), \# \left(H^0(\mathcal{F}_w, \phi[\p^\infty])\otimes_{\cO}\F_{\p}\right)\leq |\F_\p|^r.\] The result has thus been proven.
\end{proof}

\begin{proposition}\label{mu equals 0 prop}
    The following are equivalent
    \begin{enumerate}
        \item \label{p1 mu equals 0 prop} $\op{Sel}_0(\phi[\p^\infty]/\cF)$ is a cofinitely generated cotorsion $\Lambda_{\cO}$-module with $\mu_{\p}(\phi)=0$. 
        \item \label{p2 mu equals 0 prop}$\op{Sel}_0(\phi[\p]/\cF)$ is finite. 
    \end{enumerate}
    Furthermore, if the above equivalent assertions are satisfied, then, 
    \[\lambda_{\p}(\phi)\leq \op{dim}_{\kappa}\op{Sel}_0(\phi[\p]/\cF)+\sum_{w\in S(\cF)} \dim_{\F_\p}\left(H^0(\cF_{w}, \phi[\p^\infty])\otimes_{\cO} \F_{\p}\right).\]
\end{proposition}
\begin{proof}
    Let $M$ denote the Selmer group $\op{Sel}_0(\phi[\p^\infty]/\cF)$. Lemma \ref{main lemma 1} asserts that $M$ is a cofinitely generated, cotorsion as a $\Lambda_{\cO}$-module with $\mu_{\p}(M)=0$ if and only if $M[\varpi]$ is finite. On the other hand, it follows from Lemma \ref{lemma 4.3} that $M[\varpi]$ is finite if and only if $\op{Sel}_0(\phi[\p]/\cF)$ is finite. This proves that the conditions \eqref{p1 mu equals 0 prop} and \eqref{p2 mu equals 0 prop} are equivalent. In order to obtain the upper bound on $\lambda_{\p}(\phi)$ we recall from Lemma \ref{main lemma 1} that $\lambda_{\p}(\phi)\leq \dim_{\F_\p} M[\varpi]$. Note that $M[\varpi]=\op{Sel}_0(\phi[\p^\infty]/\cF)[\p]$; it follows from Lemma \ref{lemma 4.3} that 
    \[\begin{split}\dim_{\F_{\p}}M[\varpi]\leq & \dim_{\F_\p}\op{Sel}_0(\phi[\p]/\cF)+\dim_{\F_\p}\op{cok}\gamma, \\ 
    \leq & \dim_{\F_\p}\op{Sel}_0(\phi[\p]/\cF)+\sum_{w\in S(\cF)} \dim_{\F_\p}\left(H^0(\cF_{w}, \phi[\p^\infty])\otimes_{\cO} \F_{\p}\right). 
 \end{split}\]
 This concludes the proof of the result. 
\end{proof}

We demonstrate conditions under which the terms $\dim_{\F_\p}\left(H^0(\cF_{w}, \phi[\p^\infty])\otimes_{\cO} \F_{\p}\right)$ vanish.
\begin{lemma}\label{NSW lemma}
Let $G$ and $M$ be finite abelian groups of $p$-power order such that $G$ acts on $M$. Suppose that $M^G=0$, then $M=0$.
\end{lemma}
\begin{proof}
The result follows from \cite[Proposition 1.6.12]{NSW}.
\end{proof}

\begin{proposition}\label{H0 vanishing Prop}
    Assume that $v$ is a prime at which 
    \[H^0(F_v, \phi[\p])=0.\]
    Then, for all primes $w\in v(\cF)$, we have that 
    \[H^0(\cF_w, \phi[\p^\infty])=0.\]
\end{proposition}
 \begin{proof}
     Let $M:=H^0(\cF_w, \phi[\p])$ and $\Gamma :=\op{Gal}(\cF_w/F_v)$, and identify $M^{\Gamma}$ with $H^0(\cF_v, \phi[\p])$. By assumption, $M^{\Gamma}=0$. The action of $\Gamma$ is continuous and $M$ is finite. Thus, there is a finite index subgroup $\Gamma_0$ of $\Gamma$ such that $M=M^{\Gamma_0}$. It follows from Lemma \ref{NSW lemma} that $M=M^{\Gamma/\Gamma_0}=0$. In particular, this implies that \[H^0(\cF_w, \phi[\p^\infty])=0.\]
 \end{proof}

\begin{proposition}\label{MAIN prop}
    With respect to above notation, the residual Selmer group $\op{Sel}_0(\phi[\p]/\cF)$ is finite.
\end{proposition}
\begin{proof}
    Set $L:=F(\phi[\p])$ to be the field extension of $F$ in $F^{\op{sep}}$ that is fixed by the kernel of $\rho_{\phi, \p}:\op{G}_F\rightarrow \op{Aut}(\phi[\p])$. Denote by $\cL$ the constant $\Z_p$-extension of $L$ and identify $\cL$ with the composite $\cF\cdot L$. Let $H_S(\cL)$ be the maximal abelian unramified prop-extension of $\cL$ in which the all primes of $S(\cL)$ are completely split. Set $X_S(\cL)$ to denote the Galois group $\op{Gal}(H_S(\cL)/\cL)$. Theorem \ref{ROSEN thm} asserts that the $\mu$-invariant of $X_S(\cL)$ is equal to $0$. It thus follows from the structure theory of $\Lambda_{\Z_p}$-modules that $X_S(\cL)$ is finitely generated $\Z_p$-module. Let $\bar{H}_S(\cL)$ be the maximal $p$-elementary extension of $\cL$ that is contained in $H_S(\cL)$. Set $\bar{X}_S(\cL):=\op{Gal}(\bar{H}_S(\cL)/\cL)$, and identify $\bar{X}_S(\cL)$ with $X_S(\cL)/p X_S(\cL)$.
    \par Set $G:=\op{Gal}(\cL/\cF)$, note that the restriction of $\rho_{\phi, \p}$ to $\op{G}_L$ is trivial, hence, $\phi[\p]^{\op{G}_L}=\phi[\p]$. Consider the inflation-restriction sequence
    \[0\rightarrow H^1(G, \phi[\p])\xrightarrow{inf} H^1(F_S/\cF, \phi[\p])\xrightarrow{res} H^1(F_S/\cL, \phi[\p]).\] The module $\phi[\p]$ is a trivial module over $\op{G}_{\cL}$, and hence 
    \[H^1(F_S/\cL, \phi[\p])\simeq \op{Hom}(\op{Gal}(F_S/\cL), \F_\p^r).\]
    For $f\in \op{Sel}_0(\phi[\p]/\cF)$, the restriction of $f$ to $\op{Gal}(F_S/\cL)$ is unramified at all primes and split at all primes of $S(\cL)$. Therefore, $\op{res}(f)$ factors as a homomorphism of groups
    \[\op{res}(f):\bar{X}_S(\cL)\rightarrow \F_{\p}^r.\] Since $X_S(\cL)$ is finitely generated as a $\Z_p$-module, $\bar{X}_S(\cL)$ is finite. Thus, we find that $\op{res}\left(\op{Sel}_0(\phi[\p]/\cF)\right)$ is finite. On the other hand, since $G$ is finite, the cohomology group $H^1(G, \phi[\p])$ is finite. Therefore, we have shown that $\op{Sel}_0(\phi[\p]/\cF)$ is finite. 
\end{proof}
We now prove Theorem \ref{MAIN thm}.
\begin{proof}
    Part \eqref{c1 MAIN thm} and \eqref{c3 MAIN thm} of the result are a direct consequence of Proposition \ref{mu equals 0 prop} and Proposition \ref{MAIN prop}. On the other hand, part \eqref{c2 MAIN thm} is an easy consequence of part \eqref{c1 MAIN thm} and the structure theory of $\Lambda_{\cO}$-modules. In greater detail, since $Y(\phi[\p^\infty]/\cF)$ is finitely generated and torsion as a module over $\Lambda_{\cO}$ with $\mu=0$, it follows that it is pseudo-isomorphic to a free $\cO$-module of rank $\lambda_{\p}(\phi)$. Therefore, it is finitely generated as an $\cO$-module with $\cO$-rank equal to $\lambda_p(\phi)$.
\end{proof}
\subsection{An illustrative example}\label{example section} Let us illustrate Theorem \ref{MAIN thm} through an example. The Carlitz module over $F$ is the rank $1$ Drinfeld module defined by taking $\phi_T:=T+\tau$. For simplicity, we assume that $p$ is odd and $q=p$. Letting $\mathfrak{p}:=(T)$, it is easy to see that $\phi$ has good reduction at all primes and thus one can take the set $S$ in Theorem \ref{MAIN thm} to be $\{\p, \infty\}$. Thus the Theorem asserts that $\mu_{\p}(\phi)=0$ and that 
\[\begin{split}\lambda_{\p}(\phi) & \leq \op{dim}_{\F_\p}\op{Sel}_0(\phi[\p]/\cF) \\ +& \dim_{\F_\p}\left(H^0(\cF_{\p}, \phi[\p^\infty])\otimes_{\cO} \F_{\p}\right)+ \dim_{\F_\p}\left(H^0(\cF_{\infty}, \phi[\p^\infty])\otimes_{\cO} \F_{\p}\right).\end{split}\]
We identify $\F_\p$ with the field with $p$ elements. Note that in the above formula, 
\[\dim_{\F_\p}\left(H^0(\cF_{w}, \phi[\p^\infty])\otimes_{\cO} \F_{\p}\right)\leq \op{dim}_{\F_\p} \phi[\p]=1\] for $w\in \{\p, \infty\}$. The extension $\cF_\p$ is an unramified extension of $F_\p$. Since $\phi_T(x)/x=T+x^{p-1}$ is Eisenstein at $\p$, it follows that the character associated to $\phi[\p]$ is ramified at $\p$. As a result, $\phi[\p]$ is nontrivial as a module over the absolute Galois group of $\cF_\p$. Therefore, one has that 
\[\dim_{\F_\p}\left(H^0(\cF_{\p}, \phi[\p^\infty])\otimes_{\cO} \F_{\p}\right)=\dim_{\F_\p}\left(H^0(\cF_{\p}, \phi[\p])\right)=0.\] On the other hand, it follows from \cite[Theorem 7.1.13]{papibook} that $\phi[\p]$ is totally ramified at $\infty$. A similar argument therefore yields that 
\[\dim_{\F_\p}\left(H^0(\cF_{\infty}, \phi[\p^\infty])\otimes_{\cO} \F_{\p}\right)=0.\]
Thus, we have shown that 
\[\lambda_{\p}(\phi) \leq \op{dim}_{\F_\p}\op{Sel}_0(\phi[\p]/\cF).\]

\bibliographystyle{alpha}
\bibliography{references}
\end{document}